\documentclass[12pt,reqno]{amsart}
\usepackage{graphicx, color}
\usepackage{amssymb}
\usepackage{amsmath}
\usepackage{latexsym}
\usepackage{url}
\makeatletter

\newcommand{\R}{\mathbb R}
\newcommand{\Z}{\mathbb Z}

\newtheorem{thm}{Theorem}[section]
\newtheorem{lem}[thm]{Lemma}
\newtheorem{pro}[thm]{Proposition}

\newtheorem{cor}[thm]{Corollary}
\newtheorem{rk}[thm]{Remark}

\begin{document}
\title[Homoclinic classes with shadowing property]
{\bf Homoclinic classes with shadowing property}


\author{Manseob Lee
 }
\address
     { Manseob Lee :  Department of Mathematics \\
       Mokwon University, Daejeon, 302-729, Korea
      }
\email{lmsds@mokwon.ac.kr  }


\begin{abstract}
We show that for $C^1$ generic diffeomorphisms, an isolated
homoclinic class is shadowable if and only if it is a hyperbolic basic set.
\end{abstract}

\maketitle

\section{Introduction.}
In shadowing theory of dynamical systems, homoclinic classes
are basic objects in many investigations. These sets
are natural generalization of hyperbolic basic sets in Smale's
dynamical systems theory.

Let $M$ be a closed $C^{\infty}$manifold, and denote by $d$ the
distance on $M$ induced from a Riemannian metric $\|\cdot\|$ on
the tangent bundle $TM.$ Denote by ${\rm Diff}(M)$ the space of
diffeomorphisms of $M$ endowed with the $C^1$-topology.

Let $f\in{\rm Diff}(M)$. For $\delta>0$, a sequence of points
$\{x_i\}_{i=a}^{b}(-\infty\leq a< b \leq \infty)$ in $M$ is called
a {\it $\delta$-pseudo orbit} of $f$ if $d(f(x_i),
x_{i+1})<\delta$ for all $a\leq i\leq b-1$. A closed $f$-invariant set
$\Lambda\subset M$ is said to be {\it chain transitive} if for any
points $x, y\in\Lambda$ and $\delta>0$, there is a $\delta$-pseudo
orbit
$\{x_i\}_{i=a_{\delta}}^{b_{\delta}}\subset\Lambda (a_{\delta}<b_{\delta})$
of $f$ such that $x_{a_{\delta}}=x$ and $x_{b_{\delta}}=y$. For
given $x, y\in M$, we write $x\rightsquigarrow y$ if for any
$\delta>0$, there is a $\delta$-pseudo orbit
$\{x_i\}_{i=a}^{b}(a<b)$ of $f$ such that $x_a=x$ and $x_b =y.$

The set of points $\{x\in X: x\leftrightsquigarrow x\}$ is called
the {\it chain recurrent set} of $f$ and is denoted by
$\mathcal{R}(f).$ If we denote the set of periodic points $f$ by
$P(f)$, then $P(f)\subset\Omega(f)\subset\mathcal{R}(f).$ Here
$\Omega(f)$ is the non-wandering set of $f$. We write
$x\leftrightsquigarrow y$ if $x\rightsquigarrow y$ and
$y\rightsquigarrow x.$ The relation $\leftrightsquigarrow$ induces
on $\mathcal{R}(f)$ an equivalence relation whose classes are
called {\it chain components} of $f.$ Every chain component of $f$
is a closed $f$-invariant set.

Let $\Lambda\subset M$ be a closed $f$-invariant set, and denote
by $f|_{\Lambda}$ the restriction of $f$ to the set $\Lambda.$ We
say that $f|_{\Lambda}$ has the {\it shadowing property} if for
any
 $\epsilon>0$ there is $\delta>0$ such that for any
 $\delta$-pseudo orbit $\{x_i\}_{i\in\Z}\subset\Lambda$ of $f$ there is $y\in M$
 such that $d(f^i(y), x_i)<\epsilon$ for all $i\in\Z.$ We say that
 $f$ has the {\it Lipschitz shadowing property} if there exist $L>0$
 and $d>0$ such that for any $\delta(<d)$-pseudo orbit
 $\{x_i\}_{i\in\Z}$ there exists a point $y\in M$ such that
 $$d(f^i(y), x_i)<L\delta,\quad \mbox{for} i\in\Z.$$

 Note that the Lipschitz shadowing property is not
the shadowing property. In fact, Lipschitz shadowing contained in
shadowing(see \cite{P}).

It is well known that if $p$ is a hyperbolic periodic point $f$
with period $k$ then the sets $$W^s(p)=\{x\in M: f^{kn}(x)\to
p\;{\rm as}\; n\to\infty\}\;{\rm and}$$
$$W^u(p)=\{x\in M: f^{-kn}(x)\to p\;{\rm
as}\; n\to\infty\}$$ are $C^1$-injectively immersed submanifolds
of $M$. Every point $x\in W^s(p)\overline{\pitchfork} W^u(p)$ is
called a {\it homoclinic point} of $f$. The closure of the
homoclinic points of $f$ associated to $p$ is called the {\it
homoclinic class} of $f$ and it is denoted by $H_f(p).$

Note that the homoclinic class $H_f(p)$ is a subset of the chain
component $C_f(p)$. Every chain component is chain
transitive.

We consider all periodic points are the saddle type.  We say that
$p$ and $q$ are {\it homoclinically related}, and write $p\sim q$
if $W^s(p)\pitchfork W^u(q)\not=\phi,$ and $W^u(p)\pitchfork
W^s(q)\not=\phi,$ where $\pitchfork$ is transverse intersection.
Then we know that if $p\sim q$ then ${\rm index}(p)={\rm
index}(q).$

 We say that $\Lambda$ is {\it hyperbolic} if the tangent
bundle $T_{\Lambda}M$ has a $Df$-invariant splitting $E^s\oplus
E^u$ and there exists constants  $C>0$ and $0<\lambda<1$ such that
$$\|D_xf^n|_{E_x^s}\|\leq C\lambda^n\;\;{\rm and}\;\;\|D_xf^{-n}|_{E_x^u}\|\leq C\lambda^{-n} $$
for all $x\in \Lambda$ and $n\geq 0.$  It is well-known that if
$\Lambda$ is hyperbolic, then $\Lambda$ is shadowable.

We say that $\Lambda$ is {\it isolated} (or, {\it locally maximal})
if there is a compact neighborhood $U$ of $\Lambda$ such that
$\bigcap_{n\in\Z}f^n(U)=\Lambda.$

 We say that a subset $\mathcal{G}\subset {\rm Diff}(M)$ is {\it
 residual} if $\mathcal{G}$ contains the intersection of a
 countable family of open and dense subsets of ${\rm Diff}(M)$; in
 this case $\mathcal{G}$ is dense in ${\rm Diff}(M).$ A property
 "P" is said to be {\it$C^1$-generic} if "P" holds for all
 diffeomorphisms  which belong to some residual subset of ${\rm
 Diff}(M).$ We use the terminology "for $C^1$ generic $f$" to
 express "there is a residual subset $\mathcal{G}\subset {\rm
 Diff}(M)$ such that $P$ holdsfor any $f\in\mathcal{G}$".

 In \cite{AD}, Abdenur and D\'iaz has the following
 conjecture:

 \bigskip
 {\bf Conjecture.} {\it For $C^1$ generic $f,$ $f$ is shadowable if
 and only if it is hyperbolic.}

 \bigskip
Very recently, in \cite{PT} Pliyugin and Tikhomiriv proved that a
diffeomorphism has the Lipschitz shadowing is equivalent to
structural stability. By note, the conjecture is still open.

 In this paper, we give a partial answer to the above
conjecture.
 First, $C^1$-generically, chain recurrent set
is hyperbolic if and only if it is hyperbolic. Secondly,
$C^1$-generically, every shadowable homoclinic class containing a
saddle periodic point is hyperbolic if and only if it is hyperbolic.

It is explain in \cite{A} that every $C^1$-generic diffeomorphism
comes in one of two types: {\it tame diffeomorphisms}, which have a
finite number of homoclinic classes and whose nonwandering sets
admits partitions into a finite number of disjoint transitive
sets; and {\it wild diffeomorphisms}, which have an infinite
number of homoclinic classes and whose nonwandering sets admit no
such partitions. It is easy to show that if a diffeomorphism has a
finite number of chain components, then every chain component is
locally maximal, and therefore, every chain component of a tame
diffeomorphism is locally maximal. Hence we can get the following
result.
\begin{thm}For $C^1$ generic $f,$
if $f$ is tame, then the following two conditions are equivalent:
\begin{itemize}
\item[(a)]$\mathcal{R}(f)$ is hyperbolic,
\item[(b)] $\mathcal{R}(f)$ is shadowable.
\end{itemize}
\end{thm}
We say that a closed $f$-invariant set $\Lambda$ is {\it basic
set} if $\Lambda$ is isolated and $f|_{\Lambda}$ is transitive.
The main result of this paper is the following.
\begin{thm}\label{thm0} For $C^1$ generic $f,$ an isolated
homoclinic class $H_f(p)$ is shadowable if and only if $H_f(p)$ is
hyperbolic basic set.
\end{thm}
A similar result is proved in \cite{LW} for locally maximal chain
transitive sets. More precisely, it is proved that
$C^1$-generically, for any locally maximal chain transitive set
$\Lambda$, if it is shadowable then $\Lambda$ is hyperbolic.

We say that $\Lambda$ is {\it expansive} of $f$ if there is a
constant $e>0$ such that for any $x, y\in \Lambda$ if $d(f^n(x),
f^n(y))<e$ for $n\in\Z$ then $x=y.$ In \cite{YG}, Yang and Gan
proved that $C^1$-generically, every expansive homoclinic class
$H_f(p)$ is hyperbolic. Clearly, if $H_f(p)$ is hyperbolic basic
set then $H_f(p)$ is expansive, and $H_f(p)$ has the local product
structure. Thus we get:

\begin{cor} For $C^1$ generic $f,$ the followings are equivalent
\begin{itemize}
\item[{(a)}] an isolated homoclinic class $H_f(p)$ is shadowable,
\item[{(b)}] a homoclinic class $H_f(p)$ is expansive,
\item[{(c)}] $H_f(p)$ is hyperbolic.
\end{itemize}
\end{cor}
\section{Proof of Theorem \ref{thm0}}
Let $M$ and  $f\in{\rm Diff}(M)$ be as before. In this section, to
prove Theorem \ref{thm0}, we use the techniques developed by
Ma\~n\'e \cite{M}. Indeed, let $\Lambda_j(f)$ be the closure of
the set of hyperbolic periodic points of $f$ with index $j (0\leq
j\leq{\rm dim}M)$. Then if there is a $C^1$-neighborhood
$\mathcal{U}(f)$ of $f$ such that for any $g\in\mathcal{U}(f),$
any periodic points of $g$ is hyperbolic, and
$\Lambda_i(f)\cap\Lambda_j(f)=\phi$ for $0\leq i\not=j\leq{\rm
dim}M$, then $f$ satisfies both Axiom A and the no-cycle
condition. We can use the techniques to our result. To prove the
result, we will use the following Lemmas. Let $p$ be a hyperbolic
periodic point of $f\in{\rm Diff}(M).$

\begin{rk}\label{rk1}By Smale's transverse homoclinic theorem, $H_f(p)=\overline{\{q\in P(f):q\sim
p\}},$ and if $q\sim p$ then $H_f(p)=H_f(q).$
\end{rk}

\begin{lem}\label{lem0} Let $H_f(p)$ be the homoclinic class of $p$. Suppose that $f$ has the shadowing property on $H_f(p)$. Then
for any hyperbolic periodic point $q\in H_f(p),$
$$W^s(p)\cap W^u(q)\not=\phi,\quad \mbox{and}\quad W^u(p)\cap W^s(q)\not=\phi.$$
\end{lem}
\begin{proof}In this proof, we will show that $W^s(p)\cap
W^u(q)\not=\phi.$ Other case is similar. Since $p$ and $q$ are
hyperbolic saddles, there are $\epsilon(p)>0$ and $\epsilon(q)>0$
such that both $W^s_{\epsilon(p)}(p)$ and $W^u_{\epsilon(q)}(q)$
are $C^1$-embedded disks, and such that if $d(f^n (x),
f^n(p))\leq\epsilon(p)$ for $n\geq0,$ then $x\in
W^s_{\epsilon(p)}(p),$ and if $d(f^n(x), f^n (q))<\epsilon(q)$ for
$n\leq 0$ then $x\in W^u_{\epsilon(q)}(q).$

 Set $\epsilon={\rm
min}\{\epsilon(p), \epsilon(q)\},$ and let
$0<\delta=\delta(\epsilon)<\epsilon$ be the number of the
shadowing property of $f|_{H_f(p)}$ with respect to $\epsilon.$

 To
simplify, we assume that $f(p)=p$ and $f(q)=q.$  Since
$H_f(p)=\overline{\{q\in P(f):q\sim p\}}$, we can choose a $\gamma
\sim p$ such that $d(q, \gamma)<\delta/4.$ For $x\in
W^s(p)\pitchfork W^u(\gamma),$ choose $n_1>0$ and $n_2>0$ such
that $d(f^{n_1}(x), p)<\delta/4$ and $d(f^{-n_2}(x),
\gamma)<\delta/4.$ Thus $d(f^{-n_2}(x), q)<d(f^{-n_2}(x),
\gamma)+d(\gamma, q)<\delta/2.$ Therefore, we can make the
following pseudo orbit:
$$\xi=\{\dots, p, f^{n_1-1}(x), \ldots, f(x), x, f^{-1}(x),
\ldots, f^{-n_2+1}(x), q, \dots\}.$$ Then clearly, $\xi\subset
H_f(p).$ Since $f$ has the shadowing property on $H_f(p),$ choose
a point $y\in M$ such that $d(f^{i}(y), x_i)<\epsilon,$ for
$i\in\Z.$ Thus $f^{n_1+l}(y)\in W^s_{\epsilon}(p),$ and
$f^{-n_2-l}(y)\in W^u_{\epsilon}(q)$ for $l>0.$ Therefore, $y\in
f^{-n_1-l}(W^s_{\epsilon}(p))\cap f^{n_2+l}(W^u_{\epsilon}(q)).$
This means $y\in W^s(p)\cap W^u(q).$ Thus $W^s(p)\cap
W^u(q)\not=\phi.$
\end{proof}
\begin{lem}\label{lem1} There is a residual set $\mathcal{G}_1\subset
{\rm Diff}(M)$ such that $f\in\mathcal{G}_1$ satisfies the
following properties:
\begin{itemize}
\item[{(a)}] Every periodic point of $f$ is hyperbolic and all
their invariant manifolds are transverse (Kupka-Smale).
\item[{(b)}] $C_f(p)=H_f(p)$, where $p$ is a hyperbolic periodic point (\cite{BC}).
\end{itemize}
\end{lem}

\begin{lem}\label{lem2}  There is a residual set
$\mathcal{G}_2\subset {\rm Diff}(M)$ such that if
$f\in\mathcal{G}_2,$ and $f$ has the shadowing property on
$H_f(p),$ then for any hyperbolic periodic point $q\in H_f(p),$
$$W^s(p)\pitchfork W^u(q)\not=\phi\quad\mbox{and}\quad W^u(p)\pitchfork
W^s(q)\not=\phi.$$
\end{lem}
\begin{proof}Let $\mathcal{G}_2$ be Lemma \ref{lem2}(a), and let$f\in\mathcal{G}_2.$ Let $q\in H_f(p)\cap
P(f)$ be a hyperbolic saddle. Then $W^s(p)$ and $W^u(q)$ are
transverse, and $W^u(p)$ and $W^s(q)$  are also transverse.
Thus $$W^s(p)\pitchfork W^u(q)\not=\phi\quad\mbox{and}\quad
W^u(p)\pitchfork W^s(q)\not=\phi.$$
\end{proof}

\begin{pro}\label{pro1} For $C^1$ generic $f,$ if an isolated homoclinic class $H_f(p)$ is
shadowable, then there exist constants $m>0$ and $0<\lambda<1$
such that for any periodic point $p\in\Lambda,$
$$\prod_{i=0}^{\pi(p)-1}\|Df^m|_{E^s(f^{im}(p))}\|<\lambda^{\pi(p)},$$
$$\prod_{i=0}^{\pi(p)-1}\|Df^{-m}|_{E^s(f^{-im}(p))}\|<\lambda^{\pi(p)}
$$
and
$$\|Df^m|_{E^s(p)}\|\cdot\|Df^{-m}|_{E^u(f^{-m}(p))}\|<\lambda^2,$$
where $\pi(p)$ denotes the period of $p.$
\end{pro}

From now, we use the following notion in \cite{YG}.
 For $\eta>0$ and $f\in{\rm Diff}(M)$, a $C^1$ curve $\gamma$ is
called {\it $\eta$-simply periodic curve} of $f$ if
\begin{itemize}
\item $\gamma$ is diffeomorphic to $[0, 1]$ and it two end points
are hyperbolic periodic points of $f,$
\item $\gamma$ is periodic with period $\pi(\gamma)$, i.e.,
$f^{\pi(\gamma)}(\gamma)=\gamma,$ and $l(f^i(\gamma))<\eta$ for
any $0\leq i\leq\pi(\gamma)-1,$ where $l(\gamma)$ denotes the
length of $\gamma.$
\item $\gamma$ is normally hyperbolic.
\end{itemize}

Let $p$ be a periodic point of $f.$ For $\delta\in (0, 1)$, we say
$p$ has a {\it $\delta$-weak eigenvalue} if $Df^{\pi(p)}(p)$ has
an eigenvalue $\mu$ such that
$$(1-\delta)^{\pi(p)}<\mu<(1+\delta)^{\pi(p)}.$$

 \begin{lem}\label{lem3} (\cite{YG}) There is a residual set
$\mathcal{G}_3\subset {\rm Diff}(M)$ such that any
$f\in\mathcal{G}_3$ and hyperbolic periodic
 point $p$ of $f$, we have:
 \begin{itemize}
\item[{(a)}] for any $\eta>0$, if for any $C^1$ neighborhood
$\mathcal{U}(f)$ of $f,$ some $g\in\mathcal{U}(f)$ has an
$\eta$-simply periodic curve $\gamma,$ such that two endpoints of
$\gamma$ are homoclinic related with $p_g,$ then $f$ has an
$2\eta$-simply periodic curve $\alpha$ such that the two endpoints
of $\alpha$ are homoclinically related with $p.$

\item[{(b)}] for any $\delta>0$ if for any $C^1$ neighborhood
$\mathcal{U}(f)$ of $f$, some $g\in\mathcal{U}(f)$ has a periodic
point $q\sim p_g$ with $\delta$-weak eigenvalue, then $f$ has a
periodic point $q'\sim p$ with $2\delta$-weak eigenvalue.

\item[{(c)}] for any $\delta>0$ if for any $C^1$-neighborhood
$\mathcal{U}(f)$ some $g\in\mathcal{U}(f)$ has a periodic point
$q\sim p_g$ with $\delta$-weak eigenvalue and every eigenvalue of
$q$ is real, then $f$ has a periodic point $q'\sim p$ with
$2\delta$-weak eigenvalue and every eigenvalue of $q'$ is real.

\item[{(d)}] for any $\delta>0$, $f$ has a periodic point $q\sim
p$ with $\delta$-weak eigenvalue, then $f$ has a periodic point
$q'\sim p$ with $\delta$-weak eigenvalue, whose eigenvalue are all
real.
 \end{itemize}
 \end{lem}
 The following lemma says that the map $f\mapsto C_f(p)$ is upper
 semi-continuous.
\begin{lem}\label{lem4} For any $\epsilon>0,$ there is $\delta>0$
such that if $d_1(f, g)<\delta$ then $C_g(p_g)\subset
B_{\epsilon}(C_f(p)),$ where $d_1$ is the $C^1$-metric on ${\rm
Diff(M)}.$
\end{lem}
\begin{proof} See [\cite{S1}, Lemma].
\end{proof}

Let $H_f(p)$ be the homoclinic class of $p.$ It is known that the
map $f\mapsto H_f(p)$ is lower semi-continuous. Thus by Lemma
\ref{lem1}(b), there is a $C^1$-residual set in ${\rm Diff}(M)$
such that for any $f$ in the set, the map $f\mapsto
H_f(p)(=C_f(p))$ is semi-continuous.

It is known that $C^1$ generically, a homoclinic class
$H_f(p)$ is the chain component $C_f(p).$

\begin{rk}\label{rk3}There is a residual set $\mathcal{G}_4\subset{\rm Diff}(M)$ such that for any $f\in\mathcal{G}_3$,
we have the following property. Let $C_f(p)$ be isolated in
$U.$ Then if $C_f(p)$ is a semi-continuous, then for any
$\epsilon>0$, there is $\delta>0$ such that if $d_1(f,
g)<\delta(g\in {\rm Diff}(M))$ then $C_g(p_g)\subset
B_{\epsilon}(C_f(p)),$ and $C_f(p)\subset B_{\epsilon}(C_g(p_g)),$
where $d_1$ is the $C^1$ metric on $M.$
\end{rk}
 \begin{lem}\label{lem5} There is a residual set
$\mathcal{G}_5\subset {\rm Diff}(M)$ such that if
$f\in\mathcal{G}_5,$ and an isolated homoclinic class $H_f(p)$ is
 shadowable, then there is a $\delta>0$
 such that for any periodic point $q\sim p,$ $q$ has no
 $\delta$-weak eigenvalue.
 \end{lem}
\begin{proof} Let $\mathcal{G}_5=\mathcal{G}_1\cap
\mathcal{G}_2\cap \mathcal{G}_3\cap\mathcal{G}_4,$ and let
$f\in\mathcal{G}_5.$ Assume that $f$ has the shadow property on
$H_f(p).$ We will derive a contradiction. For any $\delta>0$ there
is a hyperbolic periodic point $q\in H_f(p)$ such that $q\sim p$
with $\delta$-weak eigenvalue. From \cite{S}(Theorem B), let
$\eta>0$ be sufficiently small. Then for any $C^1$-neighborhood
$\mathcal{U}(f)$ of $f$, there is $g\in\mathcal{U}(f),$ $g$ has an
$\eta$-simply periodic curve $\mathcal{I}_q,$ whose endpoints are
homoclinically related to $p_g,$ and $\eta$-simply periodic curve
$\mathcal{I}_q,$ is in $C_g(p_g).$
 Then
we know that for some $l>0,$ $\mathcal{I}_q,$ is a $g^l$-invariant
small curve containing $q$(see \cite{S}, $q$ is the center of
$\mathcal{I}_q$), where $q\sim p_g.$

 By Lemma \ref{lem3}(a), for given $\eta>0$ $f$ has a $2\eta$-simply
periodic curve $\mathcal{J}_q$ such that the endpoints of
$\mathcal{J}_q$ are homoclinically related to $p,$ and
$\mathcal{J}_q$ contains $q\sim p.$ By Remark \ref{rk3} and
$C_f(p)$ is isolated in $U,$ we know that $\mathcal{J}_q\subset
U.$ Then for some $l'>0,$ $\mathcal{J}_q$ is a $f^{l'}$-invariant
small curve center at $q$. To simplify, we denote $f^{l'}$ by $f$.
Since $C_f(p)$ is an isolated,
$$\mathcal{J}_q\subset
C_f(p)=\bigcap_{n\in\Z}f^n(U).$$ Since $f$ has the shadowing
property on $C_f(p)$, $f$ must have the shadowing property on
$\mathcal{J}_q$. But it is contradiction. Thus there is a
$C^1$-neighborhood $\mathcal{U}(f)$ of $f$ such that for any
$g\in\mathcal{U}(f)$ any periodic point $q\sim p_g$ has no
$\delta/2$-weak eigenvalue.
\end{proof}

 Let $p$ be a hyperbolic periodic point $f\in{\rm Diff}(M).$
\begin{rk}\label{rk4}\cite{YG}, There is a residual set $\mathcal{G}_6\subset{\rm Diff}(M)$ such that
for any $f\in\mathcal{G}_6,$ and any $\delta>0$ if every periodic
point $q\sim_f p$ has no $2\delta$-weak eigenvalue, then there is
a $C^1$ neighborhood $\mathcal{U}(f)$ of $f$ such that for any
$g\in\mathcal{U}(f)$ any periodic point $q\sim p_g$ has no
$\delta$-weak eigenvalue.
\end{rk}
 {\bf  Proof of Proposition \ref{pro1}.} Let
$f\in\mathcal{G}_5\cap\mathcal{G}_6.$ Assume that $f$ has the
shadowing property on $H_f(p).$ By Lemma \ref{lem3}, and Remark
\ref{rk4}, there is a $C^1$ neighborhood $\mathcal{U}(f)$ of $f$
such that for any $g\in\mathcal{U}(f)$ any periodic point $q\sim
p_g$ has no $\delta/2$ weak eigenvalue.

Thus from the extension of Franks' lemma (\cite{G}) and
Ma\~n\'e(\cite{M}), any small perturbation of the derivative
along a periodic orbit,
there exists a small perturbation of the underlying diffeomorphism
which preserves the homoclinic relation simultaneously. Since
there is a $C^1$ neighborhood $\mathcal{U}(f)$ of $f$ such that
for any $g\in\mathcal{U}(f)$, any periodic point $q\sim p_g$ has
no $\delta/2$ weak eigenvalue. Therefore, from the extension of
Franks' lemma
$$\{D_q f, D_{f(q)}f, \ldots, D_{f^{\pi(q)-1}(q)}f: q\sim p\}$$ is an
uniformly hyperbolic family of periodic sequences of isomorphisms
of $\R^{{\rm dim}M}.$ Thus we get Proposition \ref{pro1}.

\bigskip

 {\bf  Proof of Theorem \ref{thm0}.} Let $f\in\mathcal{G}_5\cap\mathcal{G}_6,$ and let $C_f(p)$ be an
 isolated in $U.$ Assume that $f$ has the shadowing property on
 $C_f(p).$ Then $C_f(p)$ satisfies Proposition \ref{pro1}. Since $f$
 has the shadowing property on $C_f(p)$, by \cite{WGW},  $C_f(p)$
 is hyperbolic. Thus $C^1$-generically, $H_f(p)$ is hyperbolic
 basic set.

\end{document}